\documentclass[12pt,reqno]{amsart}
\usepackage{fullpage}
\usepackage{times}
\usepackage{amsmath,amsthm,amsfonts,amssymb,enumerate,mathrsfs,setspace,verbatim,latexsym}

\newtheorem{theorem}{Theorem} 
\newtheorem{lemma}[theorem]{Lemma}
\newtheorem{prop}[theorem]{Proposition}

\newtheorem{conj}[theorem]{Conjecture}
\theoremstyle{definition}
\newtheorem{definition}[theorem]{Definition}

\renewcommand{\mod}[1]{{\ifmmode\text{\rm\ (mod~$#1$)}\else\discretionary{}{}{\hbox{ }}\rm(mod~$#1$)\fi}}
\newcommand{\ep}{\varepsilon}

\newcommand{\intg}{\mathbb Z}

\newcommand{\PP}{\mathcal P}
\newcommand{\li}{\mathop{\rm li}}

\begin{document}
\thispagestyle{empty}

\title{The size of coefficients of certain polynomials related to the Goldbach conjecture}
\author{Greg MARTIN}
\author{Charles L. SAMUELS}
\address{Department of Mathematics \\ University of British Columbia \\ Room 121, 1984 Mathematics Road \\ Canada V6T 1Z2}
\email{gerg@math.ubc.ca, csamuels@math.ubc.ca}
\subjclass[2000]{11A41, 11N37}

\begin{abstract}
Recent work of Borwein, Choi, and the second author examined a collection of polynomials closely related to the Goldbach conjecture: the polynomial $F_N$ is divisible by the $N$th cyclotomic polynomial if and only if there is no representation of $N$ as the sum of two odd primes. The coefficients of these polynomials stabilize, as $N$ grows, to a fixed sequence $a(m)$; they derived upper and lower bounds for $a(m)$, and an asymptotic formula for the summatory function $A(M)$ of the sequence, both under the assumption of a famous conjecture of Hardy and Littlewood. In this article we improve these results: we obtain an asymptotic formula for $a(m)$ under the same assumption, and we establish the asymptotic formula for $A(M)$ unconditionally.
\end{abstract}

\maketitle

\section{Introduction}

Let $R(n)$ denote the number of representations of $n$ as the sum of two odd primes.  That is, $R(n)$ is the number of ordered pairs $(p,q)$ of odd primes satisfying $p+q=n$. 
Of course $R(n)=0$ when $n$ is odd, while the celebrated Goldbach conjecture is equivalent to the statement that $R(n)\ge1$ for all even integers $n\ge6$. Subsequently, define
\[
a(m) = \sum_{d\mid m} R(d);
\]
these quantities are closely related to a sequence of polynomials, which we describe shortly, that have a surprising connection to the Goldbach conjecture. Also define
\[
A(M) = \sum_{m=1}^M a(2m) = \sum_{m=1}^{2M} a(m),
\]
a summatory function that encodes the average behavior of $a(m)$.

The purpose of this paper is to establish two theorems concerning the sizes of $A(M)$ and $a(m)$ that improve results obtained by Borwein, Choi, and the second author in~\cite{BCS}. The first of these theorems is an asymptotic formula for $A(M)$.

\begin{theorem} \label{UnconditionalABound}
  For all $M\ge3$,
  \begin{equation*}
    A(M) = \frac{\pi^2M^2}{3\log^2M} + O\left( \frac{M^2\log\log M}{\log^3M} \right).
  \end{equation*}
\end{theorem}

We emphasize that this theorem is unconditional; by contrast, the authors of~\cite{BCS} established this asymptotic formula without an explicit error term, but only under the assumption of a well-known conjecture on the number of Goldbach representations of an integer $n$:

\begin{conj}[Hardy and Littlewood~\cite{HL}] \label{HLConj}
As $n$ tends to infinity,
\begin{equation}
  R(2n) \sim 2C_2 \frac{n}{\log^2n} \prod_{\substack{p|n \\ p > 2}}\frac{p-1}{p-2},
\end{equation}
where $C_2$ is the twin primes constant
\begin{equation*}
	C_2 = \prod_{p>2}\left( 1-\frac{1}{(p-1)^2}\right).
\end{equation*}
\end{conj}

The authors of~\cite{BCS} do obtain an unconditional lower bound on $A(M)$, namely
\begin{equation} \label{WeakLower}
    A(M) \geq M\log M + O(M).
\end{equation}
This lower bound required the use of a deep result of Montgomery and Vaughan~\cite{MV} on the exceptional set in the Goldbach conjecture, while our proof of Theorem~\ref{UnconditionalABound} is elementary, with the deepest ingredient being the prime number theorem. The surprising gap between the asymptotic formula in Theorem~\ref{UnconditionalABound} and the lower bound~\eqref{WeakLower} can be explained by the fact that the authors of~\cite{BCS} actually prove the much stronger result
 \begin{equation*}
    \sum_{m=1}^{2M} \#\big\{ d\mid m\colon R(d) \ge 1 \big\} \ge M\log M + O(M),
 \end{equation*}
which does indeed imply~\eqref{WeakLower}, since the summand on the left-hand side is at most $\sum_{d\mid m} R(d) = a(m)$.

Our second theorem, an asymptotic formula for $a(m)$ conditional on the aforementioned conjecture of Hardy and Littlewood, is best stated after defining the following multiplicative function.

\begin{definition}
\label{Fdef}
$J(m)$ is the multiplicative function given by the following formula: if $2^k\parallel m$, then
\[
J(m) = \bigg( 2-\frac1{2^k} \bigg) \prod_{\substack{p^\ell\parallel m \\ p>2}} \bigg( 1-\frac2{p^{\ell+1}} \bigg) \bigg( 1-\frac2p \bigg)^{-1}.
\]
Here, $p^\ell \parallel m$ means that $p^\ell\mid m$ but $p^{\ell+1}\nmid m$.
\end{definition}

\begin{theorem} \label{HLConjResult}
If Conjecture~\ref{HLConj} is true, then $$a(2m) \sim \frac{2C_2 J(m) m}{\log^2m}$$ as $m$ tends to infinity.
\end{theorem}

The authors of~\cite{BCS} were able to derive from Conjecture~\ref{HLConj} the upper and lower bounds
\begin{equation} \label{BasicaAsym}
	\frac{2C_2m}{\log^2m} \lesssim a(2m) \lesssim \frac{4C_2m}{\log^2 m} \prod_{\substack{p^\ell\parallel m \\ p>2}} \bigg( 1-\frac2{p^{\ell+1}} \bigg) \bigg( 1-\frac2p \bigg)^{-1};
\end{equation}
we are able here to close the small gap between these bounds.

The function $a(m)$ and the Goldbach conjecture are linked via the sequence of polynomials
\begin{equation*}
       F_N(z) = \sum_{k=0}^{N-1} \bigg( \sum_{n=1}^{N-1}\chi_\PP(n)z^{kn} \bigg)^2,
\end{equation*}
where
\begin{equation*}
      \chi_\PP(n) = \begin{cases}
          1, & \text{if $n$ is an odd prime,} \\
          0, & \text{otherwise}.
          \end{cases}
\end{equation*}
For example,
\begin{multline*}
F_{10}(z) = 9 + (z^7+z^5+z^3)^2 +(z^{14}+z^{10}+z^6)^2 +(z^{21}+z^{15}+z^9)^2  \\
+(z^{28}+z^{20}+z^{12})^2 +(z^{35}+z^{25}+z^{15})^2 +(z^{42}+z^{30}+z^{18})^2 \\
+(z^{49}+z^{35}+z^{21})^2 +(z^{56}+z^{40}+z^{24})^2 + (z^{63}+z^{45}+z^{27})^2.
\end{multline*}
It is not hard to see that for $m\ge1$, the coefficient of $z^m$ in $F_N(z)$ is a nonnegative integer that is at most $a(m)$, and in fact it equals $a(m)$ for all $N\ge m$. For example, when expanded out
\[
F_{10}(z) = 9 + z^6 + 2 z^8 + 3 z^{10} + \cdots + z^{126},
\]
reflecting the first ten values
\[
\big( a(1), \dots, a(10) \big) =  ( 0,0, 0, 0, 0, 1, 0, 2, 0, 3 ).
\]
In other words, the sequence of polynomials $F_N(z)-F_N(0)$ converges coefficient-wise to the fixed formal power series $\sum_{m=1}^\infty a(m)z^m$.

Letting $\Phi_k(z)$ denote the $k$th cyclotomic polynomial as usual, the authors of~\cite{BCS} show that $F_{2N}(z)$ is divisible by $\Phi_{4N}(z)$ for every positive integer $N$. Experimental evidence suggests:

\begin{conj}[Borwein, Choi, and Samuels] \label{GeneralGoldbach}
For every integer $N\ge3$, the polynomial $F_{2N}(z)/\Phi_{4N}(z)$ is irreducble in $\intg[z]$.
\end{conj}

\noindent The relationship between $F_N$ and the Goldbach conjecture is more than superficial, however, as the following startling theorem displays:

\begin{theorem}[Borwein, Choi, and Samuels] \label{ImpliesGoldbach}
$\Phi_N(z)$ divides $F_N(z)$ if and only if there is no representation of $N$ as the sum of two odd primes. In particular, Conjecture~\ref{GeneralGoldbach} implies the Goldbach conjecture.
\end{theorem}

\section{Proofs of our results}

We begin by proving Theorem~\ref {UnconditionalABound}, although first we need to devote some time to a technical lemma that counts the number of pairs of primes whose sum lies below a given bound. Afterwards, we derive Theorem~\ref{HLConjResult} from Proposition~\ref{using epsilons prop} below.

In order to establish Theorem \ref{UnconditionalABound}, we must first study the function
\begin{equation*}
Q(x) = \sum_{p+q\le x} 1,
\end{equation*}
where $p$ and $q$ always denote primes in this paper.

\begin{lemma} \label{Q lemma}
  Uniformly for $x\ge3$,
  \begin{equation*}
    Q(x) = \frac{x^2}{2\log^2 x} + O\left( \frac{x^2\log\log x}{\log^3 x} \right).
  \end{equation*}
\end{lemma}

\begin{proof}
We begin by writing
\[
Q(x) = \sum_{p\le x} \pi(x-p) = \sum_{x/\log x \le p \le x-\sqrt x} \pi(x-p) + O\bigg( \sum_{p\le x/\log x} \pi(x-p) + \sum_{x-\sqrt x\le p\le x} \pi(x-p) \bigg).
\]
Trivially $\pi(x-p) \le \pi(x) \le x$, so
\begin{align}
  Q(x) &= \sum_{x/\log x \le p \le x-\sqrt x} \pi(x-p) + O\bigg( \sum_{p\le x/\log x} \pi(x) + \sum_{x-\sqrt x\le p\le x} x \bigg) \notag \\
  &= \sum_{x/\log x \le p \le x-\sqrt x} \pi(x-p) + O\bigg( \pi(x) \pi\bigg( \frac x{\log x} \bigg) + x\sqrt x \bigg) \notag \\
  &= \sum_{x/\log x \le p \le x-\sqrt x} \pi(x-p) + O\bigg( \frac{x^2}{\log^3x} \bigg).
  \label{deal with main term}
\end{align}
In the main term, the prime number theorem gives
\[
\sum_{x/\log x\le p\le x-\sqrt x} \pi(x-p) = \sum_{x/\log x \le p \le x-\sqrt x} \bigg( \li(x-p) + O\bigg( \frac{x-p}{\log^2(x-p)} \bigg) \bigg)
\]
(we could insert a better error term, but it would not improve the final result). Since $x-p \ge \sqrt x$, we have $\log(x-p) \gg \log x$ and so
\begin{align*}
&= \sum_{x/\log x \le p \le x-\sqrt x} \li(x-p) + O\bigg( \sum_{x/\log x \le p \le x-\sqrt x} \frac x{\log^2x} \bigg) \\
&= \sum_{x/\log x \le p \le x-\sqrt x} \li(x-p) + O\bigg( \frac x{\log^2x} \pi(x) \bigg) \\
&= \sum_{x/\log x \le p \le x-\sqrt x} \li(x-p) + O\bigg( \frac{x^2}{\log^3x} \bigg),
\end{align*}
which transforms equation~\eqref{deal with main term} into
\begin{equation}
Q(x) = \sum_{x/\log x \le p \le x-\sqrt x} \li(x-p) + O\bigg( \frac{x^2}{\log^3x} \bigg).
  \label{li sum}
\end{equation}

Using partial summation, we have
\begin{align*}
  \sum_{x/\log x \le p \le x-\sqrt x} & \li(x-p) = \int_{x/\log x}^{x-\sqrt x} \li(x-t) \,d\pi(t) \\
  &= \pi(x-\sqrt x)\li(\sqrt x) - \pi\bigg( \frac x{\log x} \bigg) \li\bigg( x - \frac x{\log x} \bigg) + \int_{x/\log x}^{x-\sqrt x} \frac{\pi(t)}{\log(x-t)} \,dt,
\end{align*}
since the $t$-derivative of $\li(x-t)$ is $-1/\log(x-t)$. In other words,
\begin{align*}
  \sum_{x/\log x \le p \le x-\sqrt x} \li(x-p) &= O\bigg( x\sqrt x + \pi\bigg( \frac x{\log x} \bigg) \li(x) \bigg) + \int_{x/\log x}^{x-\sqrt x} \frac{\pi(t)}{\log(x-t)} \,dt \\
  &= \int_{x/\log x}^{x-\sqrt x} \frac{\pi(t)}{\log(x-t)} \,dt + O\bigg( \frac{x^2}{\log^3x} \bigg),
\end{align*}
and so equation~\eqref{li sum} becomes
\[
Q(x) = \int_{x/\log x}^{x-\sqrt x} \frac{\pi(t)}{\log(x-t)} \,dt + O\bigg( \frac{x^2}{\log^3x} \bigg).
\]
Using the prime number theorem again, this becomes
\begin{align}
  Q(x) &= \int_{x/\log x}^{x-\sqrt x} \frac1{\log(x-t)}\bigg( \frac t{\log t} + O\bigg( \frac t{\log^2t} \bigg) \bigg) \,dt + O\bigg( \frac{x^2}{\log^3x} \bigg) \notag \\
  &= \int_{x/\log x}^{x-\sqrt x} \frac t{(\log t)\log(x-t)} \,dt + O\bigg( \int_{x/\log x}^{x-\sqrt x} \frac t{(\log^2t) \log(x-t)}\,dt + \frac{x^2}{\log^3x} \bigg).
  \label{two integrals}
\end{align}
In the error term, again $\log(x-t) \gg \log x$ and $\log^2t \gg \log^2x$ due to the endpoints of integration, and so the entire integral is $\ll x^2/\log^3x$. In the main term, we have
\[
\log x \ge \log t \ge \log \frac x{\log x} = \log x - \log\log x = (\log x) \bigg( 1+O\bigg( \frac{\log\log x}{\log x} \bigg) \bigg),
\]
and therefore equation~\eqref {two integrals} becomes
\begin{equation}
  Q(x) = \frac1{\log x}\bigg( 1+O\bigg( \frac{\log\log x}{\log x} \bigg) \bigg) \int_{x/\log x}^{x-\sqrt x} \frac t{\log(x-t)} \,dt + O\bigg( \frac{x^2}{\log^3x} \bigg).
  \label{one integral}
\end{equation}
Finally,
\begin{align}
  \int_{x/\log x}^{x-\sqrt x} \frac t{\log(x-t)} \,dt &= \int_0^{x-2} \frac t{\log(x-t)} \,dt + O\bigg( \int_0^{x/\log x} t \,dt + \int_{x-\sqrt x}^{x-2} t \,dt \bigg) \notag \\
  &= \int_2^x \frac{x-u}{\log u} \,du + O\bigg( \frac{x^2}{\log^2x} \bigg) \notag \\
  &= x\li(x) - \int_2^x \frac u{\log u}\,du + O\bigg( \frac{x^2}{\log^2x} \bigg).
  \label{strange li}
\end{align}
By integration by parts, this integral is
\begin{align*}
  \int_2^x \frac u{\log u}\,du &= \frac{u^2}2 \frac1{\log u} \bigg|_2^x + \int_2^x \frac{u^2}2 \frac1{u\log^2 u} \,du \\
  &= \frac{x^2}{2\log x} + O\bigg( 1 + \int_2^{\sqrt x} \frac u{\log^2 u}\,du + \int_{\sqrt x}^x \frac u{\log^2u} \,du \bigg) \\
  &= \frac{x^2}{2\log x} + O\bigg( \sqrt x \cdot x + x \frac x{\log^2 x} \bigg) = \frac{x^2}{2\log x} + O\bigg( \frac{x^2}{\log^2 x} \bigg).
\end{align*}
Therefore equation~\eqref{strange li} becomes
\[
\int_{x/\log x}^{x-\sqrt x} \frac t{\log(x-t)} \,dt = x\li(x) - \frac{x^2}{2\log x} + O\bigg( \frac{x^2}{\log^2x} \bigg) = \frac{x^2}{2\log x} + O\bigg( \frac{x^2}{\log^2x} \bigg)
\]
by the fact that $\li(x) = x/\log x + O(x/\log^2 x)$. Using this in equation~\eqref{one integral} finally yields
\begin{align*}
  Q(x) &= \frac1{\log x}\bigg( 1+O\bigg( \frac{\log\log x}{\log x} \bigg) \bigg) \bigg( \frac{x^2}{2\log x} + O\bigg( \frac{x^2}{\log^2x} \bigg) \bigg) + O\bigg( \frac{x^2}{\log^3x} \bigg) \\
  &= \frac{x^2}{2\log^2 x} + O\bigg( \frac{x^2\log\log x}{\log^3x} \bigg),
\end{align*}
as claimed.
\end{proof}

Equipped with Lemma \ref{Q lemma}, we are now prepared to prove Theorem~\ref{UnconditionalABound}.

\begin{proof}[Proof of Theorem \ref{UnconditionalABound}]
Starting with the definitions of $a(m)$ and $A(M)$, we have
\begin{equation*}
A(M) = \sum_{m=1}^{2M} a(m) = \sum_{m=1}^{2M} \sum_{d\mid m} R(d) = \sum_{m=1}^{2M} \sum_{d\mid m} \sum_{p+q=d} 1 = \sum_{p+q \le 2M} \sum_{\substack{1\le m\le 2M \\ (p+q)\mid m}} 1.
\end{equation*}
Writing $m=(p+q)n$, we obtain
\begin{equation}
A(M) = \sum_{p+q \le 2M} \sum_{1\le n\le 2M/(p+q)} 1 = \sum_{1\le n\le M/2} \sum_{p+q \le 2M/n} 1 = \sum_{1\le n\le M/2} Q\bigg( \frac{2M}{p+q} \bigg).
\label{telescoped}
\end{equation}
The trivial bound $Q(x) \le x^2$ allows us to write
\[
A(M) = \sum_{1\le n\le \log^3M} Q\bigg( \frac{2M}n \bigg) + O\bigg( \sum_{n > \log^3M} \bigg( \frac{2M}n \bigg)^2 \bigg) = \sum_{1\le n\le \log^3M} Q\bigg( \frac{2M}n \bigg) + O\bigg( \frac{M^2}{\log^3M} \bigg),
\]
since $\sum_{n>\log^3M} n^{-2} \ll 1/\log^3M$ by comparison with an integral. We use Lemma~\ref{Q lemma} to get
\begin{align*}
A(M) &= \sum_{1\le n\le \log^3M} \bigg( \frac{(2M/n)^2}{2\log^2(2M/n)} + O\bigg( \frac{(2M/n)^2\log\log(2M/n)}{\log^3(2M/n)} \bigg) \bigg) + O\bigg( \frac{M^2}{\log^3M} \bigg) \\
&= 2M^2 \sum_{1\le n\le \log^3M} \frac1{\log^2(2M/n)}\frac1{n^2} + O\bigg( \sum_{1\le n\le\log^3M} \frac{\sqrt{2M}\log\log2M}{\log^32M} \bigg( \frac{2M}n \bigg)^{3/2} + \frac{M^2}{\log^3M} \bigg),
\end{align*}
since $\sqrt x\log\log x/\log^3x$ is an (eventually) increasing function of~$x$. By the convergence of $\sum_n n^{-3/2}$, we obtain
\[
A(M) = 2M^2 \sum_{1\le n\le \log^3M} \frac1{\log^2(2M/n)}\frac1{n^2} + O\bigg( \frac{M^2\log\log M}{\log^3M} \bigg).
\]
Finally, we have $\log(2M/n) = \log M - \log(n/2) = \log M + O(\log(\log^3M)) = (\log M)(1 + O(\log\log M/\log M))$ as before. Therefore
\[
A(M) = \frac{2M^2}{\log^2M} \bigg(1 + O\bigg( \frac{\log\log M}{\log M} \bigg) \bigg) \sum_{1\le n\le \log^3M} \frac1{n^2} + O\bigg( \frac{M^2\log\log M}{\log^3M} \bigg).
\]
We conclude that
\begin{align*}
A(M) &= \frac{2M^2}{\log^2M} \bigg(1 + O\bigg( \frac{\log\log M}{\log M} \bigg) \bigg) \bigg( \zeta(2) + O\bigg( \frac1{\log^3M} \bigg) \bigg) + O\bigg( \frac{M^2\log\log M}{\log^3M} \bigg) \\
&= \frac{\pi^2M^2}{3\log^2M} + O\bigg( \frac{M^2\log\log M}{\log^3M} \bigg),
\end{align*}
as desired.
\end{proof}

We now move on to a proposition from which we will deduce Theorem~\ref{HLConjResult}. Define
\[
f(n) = \prod_{\substack{p|n \\ p > 2}}\frac{p-1}{p-2}
\]
to be the multiplicative function appearing in Conjecture~\ref{HLConj}, and note that if $k\ge0$ is the integer such that $2^k\parallel m$, then
\begin{align}
\sum_{d\mid m} df(d) &= \prod_{p^\ell\parallel m} \sum_{d\mid p^\ell} df(d) \notag \\
&= \prod_{p^\ell\parallel m} \big( 1 + pf(p) + p^2f(p^2) + \cdots + p^\ell f(p^\ell) \big) \notag \\
&= \bigg( 1 + 2 \frac{2^k-1}{2-1} \bigg) \prod_{\substack{p^\ell\parallel m \\ p>2}} \bigg( 1 + \frac{p-1}{p-2} \cdot p \frac{p^\ell-1}{p-1} \bigg) \notag \\
&= (2^{k+1}-1) \prod_{\substack{p^\ell\parallel m \\ p>2}} \frac{p^{\ell+1}-2}{p-2} = mJ(m)
\label{df(d) sum}
\end{align}
by comparison with Definition~\ref{Fdef}.

\begin{prop}
\label{using epsilons prop}
Let $0<\ep\le\frac12$ be given. Suppose there exists a positive integer $n(\ep)$ such that
\begin{equation}
\label{HL with epsilons}
(1-\ep)2C_2 f(n) \frac n{\log^2n} \le R(2n) \le (1+\ep)2C_2 f(n) \frac n{\log^2n}
\end{equation}
for all $n>n(\ep)$. Then there exists a constant $m(\ep)$ such that
\begin{equation}
\label{our theorem with epsilons}
(1-2\ep)2C_2 J(m) \frac m{\log^2m} \le a(2m) \le (1+11\ep)2C_2 J(m) \frac m{\log^2m}
\end{equation}
for all $m>m(\ep)$.
\end{prop}

\noindent It is clear that Theorem~\ref{HLConjResult} follows from Proposition~\ref{using epsilons prop}, since Conjecture~\ref{HLConj} implies that the hypothesis of Proposition~\ref{using epsilons prop} holds for every $\ep>0$.

\begin{proof}[Proof of Proposition~\ref{using epsilons prop}]
We shall not keep track explicitly of the necessary value for $m(\ep)$, instead simply saying ``when $m$ is large enough'' (in terms of $\ep$) in the appropriate places. We begin by writing
\begin{equation}
\label{split at 1-epsilon}
a(2m) = \sum_{c\mid 2m} R(c) = \sum_{d\mid m} R(2d) = \sum_{\substack{d\mid m \\ d\le m^{1-\ep}}} R(2d) + \sum_ {\substack{d\mid m \\ d>m^{1-\ep}}} R(2d)
\end{equation}
(where the second equality uses the fact that $R(c)=0$ when $c$ is odd).

First we establish the upper bound in equation~\eqref{our theorem with epsilons}. We have $m^{1-\ep} > n(\ep)$ when $m$ is large enough, and so the summands in the second sum on the right-hand side of equation~\eqref{split at 1-epsilon} can be bounded above by the upper bound in equation~\eqref{HL with epsilons}. For the first sum on the right-hand side we simply use the trivial bound $R(2n) \le n$. The result is
\begin{align*}
a(2m) &\le \sum_{\substack{d\mid m \\ d\le m^{1-\ep}}} d + \sum_ {\substack{d\mid m \\ d>m^{1-\ep}}} (1+\ep)2C_2 f(d) \frac{d}{\log^2d} \\
&\le \sum_{\substack{d\mid m \\ d\le m^{1-\ep}}} m^{1-\ep} + (1+\ep)2C_2 \frac1{(1-\ep)^2\log^2m} \sum_ {\substack{d\mid m \\ d>m^{1-\ep}}} df(d) \\
&= m^{1-\ep}\tau(m) + \frac{1+\ep}{(1-\ep)^2} \frac{2C_2}{\log^2m} mJ(m)
\end{align*}
using the identity~\eqref{df(d) sum}, where $\tau(m)$ denotes the number of divisors of~$m$. It is well known that $\tau(m) \ll_\ep m^{\ep/3}$, and so the first term is less than $\ep m/\log^2m$ when $m$ is large enough. Also $(1+\ep)/(1-\ep)^2 \le 1+10\ep$ for $0<\ep\le\frac12$. Therefore
\begin{align*}
a(2m) &\le \ep \frac m{\log^2m} + (1+10\ep) \frac{2C_2}{\log^2m} mJ(m) \le (1+11\ep) 2C_2 J(m) \frac m{\log^2m}
\end{align*}
when $m$ is large enough, since $J(m)\ge1$ for all positive integers $m$ and $2C_2>1$. This establishes the upper bound in equation~\eqref{our theorem with epsilons}.

A similar method addresses the lower bound in equation~\eqref{our theorem with epsilons}. Since $m^{1-\ep} > n(\ep)$ when $m$ is large enough, the summands in the second sum on the right-hand side of equation~\eqref{split at 1-epsilon} can be bounded below by the lower bound in equation~\eqref{HL with epsilons}; the first sum on the right-hand side is nonnegative, and so we can simply delete it. We obtain the lower bound
\begin{align}
a(2m) &\ge \sum_ {\substack{d\mid m \\ d>m^{1-\ep}}} (1+\ep)2C_2 f(d) \frac{d}{\log^2d} \notag \\
&\ge (1-\ep)\frac{2C_2}{\log^2m} \sum_ {\substack{d\mid m \\ d>m^{1-\ep}}} df(d) = (1-\ep)\frac{2C_2}{\log^2m} \bigg( mJ(m) - \sum_ {\substack{d\mid m \\ d\le m^{1-\ep}}} df(d) \bigg),
\label{taking care of difference}
\end{align}
again using the identity~\eqref{df(d) sum}. This last sum is bounded above by
\begin{equation*}
\sum_ {\substack{d\mid m \\ d\le m^{1-\ep}}} df(d) \le \sum_{d\mid m} \bigg( \frac{m^{1-\ep}}d \bigg)^{1+\ep/2} df(d) \le m^{1-\ep/2} \sum_{d\mid m} \prod_{\substack{p|d \\ p > 2}}\frac{p-1}{p^{\ep/2}(p-2)}.
\end{equation*}
There are only finitely many primes $p$ for which $(p-1)/p^{\ep/2}(p-2)$ exceeds 1, and so the inner product on the right-hand side is uniformly bounded by some constant $C(\ep)$. Therefore
\[
\sum_ {\substack{d\mid m \\ d\le m^{1-\ep}}} df(d) \le C(\ep) m^{1-\ep/2} \sum_{d\mid m} 1 = C(\ep) m^{1-\ep/2} \tau(m),
\]
which as above is less than $\ep m$ for $m$ large enough. Therefore equation~\eqref{taking care of difference} becomes
\[
a(m) \ge (1-\ep)\frac{2C_2}{\log^2m}( mJ(m) - \ep m) \ge (1-2\ep) 2C_2J(m) \frac m{\log^2m}
\]
when $m$ is large enough, again since $J(m)\ge1$ always. This establishes the lower bound in equation~\eqref{our theorem with epsilons}.
\end{proof}

\section{Acknowledgment}

The first author thanks the Harish--Chandra Research Institute for its hospitality while this manuscript was being prepared. Both 
authors were supported in part by the Natural Sciences and Engineering Research Council of Canada.

\end{document}